\documentclass[a4paper,12pt]{article}
\usepackage{times}
\usepackage[english]{babel}
\usepackage{amssymb,amsmath}
\numberwithin{equation}{section}
\usepackage{amsthm}
\usepackage{array}
\usepackage{wasysym}
\usepackage[noadjust]{cite}
\usepackage{hyperref}
\usepackage[height=22.7cm , width = 16cm , top = 4cm , left = 3cm, a4paper]{geometry}
\usepackage{amssymb}
\usepackage{amsmath}
\usepackage{cite}
 \usepackage{pstricks}
\usepackage{epsfig}
\usepackage{verbatim}
\usepackage{multicol}
\usepackage{graphicx, color, psfrag}
\usepackage{graphics, psfrag}
\usepackage[numbers,sort]{natbib}
 \newtheorem{theorem}{Theorem}[section]

\theoremstyle{definition}

\newcommand{\e}{\end{document}}

\begin{document}

\thispagestyle{empty}

\author{
{{\bf Abdelfattah Mustafa, Beih S. El-Desouky   and Shamsan AL-Garash}
\newline{\it{{}  }}
 } { }\vspace{.2cm}\\
 \small \it Department of Mathematics, Faculty of Science, Mansoura University, Mansoura 35516, Egypt.
}

\title{Weibull Generalized Exponential Distribution}

\date{}

\maketitle
\small \pagestyle{myheadings}
        \markboth{{\scriptsize Weibull Generalized Exponential Distribution}}
        {{\scriptsize { Abdelfattah Mustafa, B. S. El-Desouky and Shamsan AL-Garash}}}

\hrule \vskip 8pt

\begin{abstract}
This paper introduces a new three-parameters model called the Weibull-G exponential distribution  (WGED) distribution which exhibits bathtub-shaped hazard rate. Some of it's statistical properties are obtained including quantile, moments, generating functions, reliability and order statistics. The method of maximum likelihood is used for estimating the model parameters and the observed Fisher's information matrix is derived. We illustrate the usefulness of the proposed model by applications to real data.
\end{abstract}


\vspace{ 0.2 cm}

\noindent


\section{Introduction}
\noindent
The exponential distribution (ED), \cite{Guptar12001,Gupta1999},  has a wide range of applications including life testing experiments, reliability analysis, applied statistics and clinical studies. This distribution is a special case of the two parameter Weibull distribution with the shape parameter equal to 1. The origin and other aspects of this distribution can be found in \cite{Guptar12001,Guptar22001,Guptar32002,Guptar42003}. A random variable $X$ is said to have the exponential distribution (ED) with parameters $\lambda >0$ if it's probability density function (pdf) is given by

\begin{equation} \label{eq1.1}
g(x)=\lambda e^{-\lambda x},  \quad x>0,
\end{equation}

\noindent
while the cumulative distribution function (cdf) is given by

\begin{equation} \label{eq1.2}
G(x)=1-e^{-\lambda x},  \quad x>0.
\end{equation}
\noindent

\noindent
The survival function is given by the equation

\begin{equation} \label{eq1.3}
S(x)=1-G(x)=e^{-\lambda x},  \quad x>0,
\end{equation}

\noindent
and the hazard function is

\begin{equation} \label{eq1.4}
h(x)=\lambda
\end{equation}

\noindent
Weibull distribution introduced by \cite{Weibull1951} is a popular distribution for modeling phenomenon with monotonic failure rates. But this distribution does not provide a good fit to data sets with bathtub shaped or upside-down bathtub shaped (unimodal) failure rates, often encountered in reliability, engineering and biological studies. Hence a number of new distributions modeling the data in a better way have been constructed in literature as ramifications of Weibull distribution. Bourguignon et al. \cite{Marcelo2014} introduced and studied generality a family of univariate distributions with two additional parameters, similarly as the extended Weibull,   Gurvich et al.  \cite{Gurvich1998} and gamma families,  Zografos and Balakrshnan \cite{Zografos2009}, using the Weibull generator applied to the odds ratio $\frac{G(x)}{1-G(x)}$. If $G(x)$ is the baseline cumulative distribution function (cdf) of a random variable, with probability density function (pdf) $g(x)$ and the Weibull cumulative distribution function is

\begin{equation} \label{eq1.5}
F(x;a,b)=1-e^{-a x^b}, \quad x\geq 0,
\end{equation}

\noindent
with parameters $a$ and $b$ are positive. Based on this density, by replacing $x$ with ratio $\frac{G(x)}{1-G(x)}$. The cdf of Weibull- generalized distribution, say Weibull-G distribution with two extra parameters $a$ and $b$, is defined by,  Bourguignon et al. \cite{Marcelo2014}

\begin{eqnarray} \label{eq1.6}
F(x;a,b,\lambda )&=&
\int_0^{\frac{G(x;\lambda)}{1-G(x;\lambda)}}
ab t^{b-1}e^{-a x^b}dt
\nonumber \\
&=& 1-e^{-a[\frac{G(x;\lambda)}{1-G(x;\lambda)}]^b}, \quad x\geq 0,\quad a,b \geq 0,
\end{eqnarray}

\noindent
where $G(x;\lambda)$ is a baseline cdf, which depends on a parameter $\lambda $. The corresponding family pdf becomes

\begin{equation} \label{eq1.7}
f(x;a,b,\lambda )=ab \,  g(x;\lambda) \frac{[G(x;\lambda)]^{b-1}}{[1-G(x;\lambda)]^{b+1}} e^{-a[\frac{G(x;\lambda)}{1-G(x;\lambda)}]^b}
\end{equation}

\noindent
A random variable $X$ with pdf Eq. (\ref{eq1.7}) is denoted by $X$ distributed weibll-G$(a,b,\lambda)$, $x\in R$, $a,b >0$. the additional parameters induced by the weibull generator are sought as a manner to furnish a more flexible distribution. If $b=1$, it corresponds to the exponential- generator. An interpretation of the weibull-G family of distributions can by given as follows ( Coorary, \cite{Cooray20006}) is a similar context.

\noindent
Let $Y$ be a lifetime random variable having a certain continuous $G$ distribution. The odds ratio that an individual (or component) following the lifetime $Y$ will die (failure) at time $x$ is $\frac{G(x)}{1-G(x)}$. Consider that the variability of this odds of death is represented by the random variable $X$ and assume that it  follows the Weibull model with scale $a$ and shape $b$. We can write

\begin{equation*}
Pr(Y\leq x)=Pr\left(X\leq\frac{G(x)}{1-G(x)}\right)=F(x;a,b,\lambda ).
\end{equation*}
Which is given by Eq. (\ref{eq1.6}). The survival function of the Weibull-G family is given by

\begin{equation} \label{eq1.8}
R(x;a,b,\lambda) = 1-F(x;a,b,\lambda)=e^{-a[\frac{G(x)}{1-G(x)}]^b},
\end{equation}

\noindent
and hazard rate function of the Weibull-G family is given by

\begin{eqnarray} \label{eq1.9}
h(x;a,b,\lambda) &=&
\frac{f(x;a,b,\lambda)}{1-F(x;a,b,\lambda)}=\frac{ab\,  g(x;\lambda)[G(x;\lambda)]^{b-1}}{[1-G(x;\lambda)]^{b+1}} \nonumber\\
&=& ab\,  h(x;\lambda)\, \frac{[G(x;\lambda)]^{b-1}}{[1-G(x;\lambda)]^b},
\end{eqnarray}

\noindent
where $h(x;\lambda)=\frac{g(x;\lambda)}{1-G(x;\lambda)}$. The multiplying quantity $\frac{ab\, g(x;\lambda)[G(x;\lambda)]^{b-1}}{[1-G(x;\lambda)]^b}$ works as a corrected factor for the hazard rate function of the baseline model Eq.(\ref{eq1.6}) can deal with general situation in modeling survival data with various shapes of the hazard rate function.
By using the power series for the exponential function, we obtain

\begin{equation} \label{eq1.10}
 e^{-a[\frac{G(x)}{1-G(x)}]^b}=\sum_{i=0}^{\infty}\frac{(-1)^i a^i}{i!} \bigg( \frac{G(x;\lambda)}{1-G(x;\lambda)}\bigg )^{ib},
 \end{equation}

\noindent
substituting from Eq.(\ref{eq1.10}) into Eq. (\ref{eq1.7}), we get

\begin{equation} \label{eq1.11}
f(x;a,b,\lambda )=ab\,  g(x;\lambda)
\sum_{i=0}^{\infty}\frac{(-1)^i a^i}{i!}\, \frac{[G(x;\lambda)]^{b(i+1)-1}}{[1-G(x;\lambda)]^{b(i+1)+1}}.
\end{equation}

\noindent
Using the generalized binomial theorem we have

\begin{equation} \label{eq1.12}
[1-G(x;\lambda)]^{-(b(i+1)+1)}=\sum_{j=0}^{\infty}\frac{\Gamma (b(i+1)+j+1)}{j!\Gamma (b(i+1)+1)}\, [G(x;\lambda)]^j.
 \end{equation}

Inserting Eq. (\ref{eq1.12}) in Eq. (\ref{eq1.11}), the Weibull-G family density function is

\begin{equation} \label{eq1.13}
f(x;a,b,\lambda )=\sum_{i=0}^{\infty}\sum_{j=0}^{\infty}\frac{(-1)^i a^{i+1}b\, \Gamma (b(i+1)+j+1)}{i! j!\Gamma (b(i+1)+1)}  g(x;\lambda) [G(x;\lambda)]^{b(i+1)+j-1}.
 \end{equation}

\noindent
In section 2, we define the cumulative, density and hazard functions of the Weibull-G Exponential distribution (WGED) . In section 3 and 4, we introduced the statistical properties include, quantile function skewness and kurtosis, $rth$ moments and moment generating function. The distribution of the order statistics is expressed in section 5. Finally, maximum likelihood estimation of the parameters is determined in section 6. Real data sets are analyzed in Section 7 and the results are compared with existing distributions. Finally we introduce the conclusions in Section 8.


\section{The Weibull Generalized Exponential Distribution}
In this section, we study the three parameters Weibull-G exponential distribution (WGED). Using $G(x)$ and $g(x)$ in Eq. (\ref{eq1.13}) to be the cdf and pdf of Eq. (\ref{eq1.6}) and Eq. (\ref{eq1.7}). The cumulative distribution function (cdf) of the Weibull-G exponential distribution (WGED) is given by

\begin{equation} \label{eq2.1}
F(x;a,b,\lambda ) = 1- e^{-a[e^{\lambda x}-1]^b}, \quad  x>0, \quad  a, b, \lambda   >0,
\end{equation}

\noindent
The pdf corresponding to Eq. (\ref{eq2.1}) is given by

\begin{equation} \label{eq2.2}
f(x;a,b,\lambda ) = ab\lambda \,  e^{\lambda x} [e^{\lambda x}-1]^{b-1} e^{-a[e^{\lambda x}-1]^b}, \quad  x>0,
\end{equation}

\noindent
where $a, b >0$ and $\lambda >0$ are two additional shape parameters.\\
Plots of the cdf, Eq. (\ref{eq2.1}),  of the WGED for some parameter values are displayed in Figure 1,

\begin{center}
\includegraphics[width=10cm, height=6 cm]{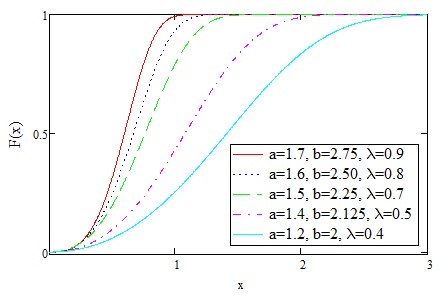}\\
Figure 1: Cumulative probability function of the WGED.
\end{center}

\noindent
We denote by $X\sim WGED (a,b,\lambda )$ a random variable having the pdf Eq. (\ref{eq2.1}). The survival function, $S(x)$, hazard rate function, $h(x)$, reversed hazard rate function, $r(x)$ and cumulative hazard rate function $H(x)$ of $X$ are given by

\begin{equation} \label{eq2.3}
S(x;a,b,\lambda  ) =1- F(x;a,b,\lambda )=e^{-a[e^{\lambda x}-1]^b}, \quad  x>0,
\end{equation}

\begin{eqnarray} \label{eq2.4}
h(x;a,b,\lambda  ) &=& ab\lambda \,  e^{\lambda x} [e^{\lambda x}-1]^{b-1}, \quad  x>0,
\end{eqnarray}

\begin{eqnarray} \label{eq2.5}
r(x;a,b,\lambda ) &=& \frac{ab\lambda \,  e^{\lambda x} [e^{\lambda x}-1]^{b-1}\cdot  e^{-a[e^{\lambda x}-1]^b}}{1-e^{-a[e^{\lambda x}-1]^b}}, \quad  x>0.
\end{eqnarray}
and
\begin{equation} \label{eq2.6}
H(x;a,b,\lambda  )= \int\limits_{0}^{\infty} h(x;a,b,\lambda )dx= a  \left[e^{e^{\lambda x}}-1 \right]^b,
\end{equation}
respectively. Plots of $S(x;a,b,\lambda ), h(x;a,b,\lambda ), r(x;a,b,\lambda )$ and $H(x;a,b,\lambda )$  of the WGED for some parameters values are displayed in Figure 2--6.
\noindent

\begin{center}
\includegraphics[width=10cm, height=6 cm]{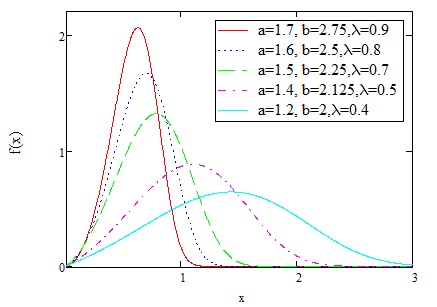}\\
Figure 2: Probability density function of the WGED.
\end{center}

\vspace{0.3 cm}

\begin{center}
\includegraphics[width=10cm, height=6 cm]{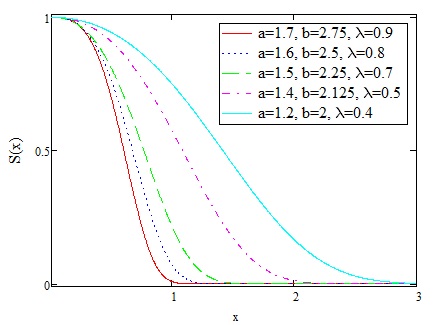}\\
Figure 3: Survival function of the WGED.
\end{center}

\vspace{ 0.3 cm}

\begin{center}
\includegraphics[width=10cm, height=6 cm]{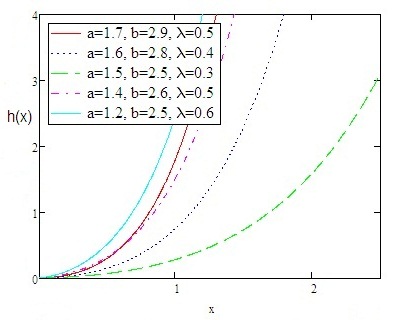}\\
Figure 4: Hazard rate function of the WGED.
\end{center}

\vspace{ 0.3 cm}

\begin{center}
\includegraphics[width=10cm, height=6 cm]{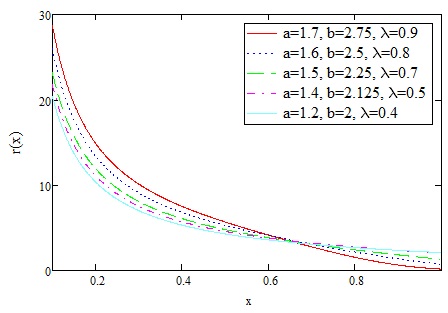}\\
Figure 5: Reversed hazard rate function of the WGED.
\end{center}

\vspace{0.3 cm}

\begin{center}
\includegraphics[width=10cm, height=6 cm]{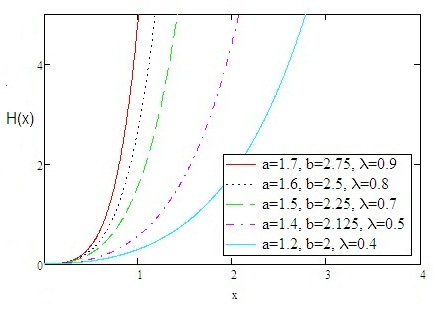}\\
Figure 6: Cumulative hazard rate function of the WGED.
\end{center}

\noindent
It is clear that the PDF and the hazard function have many different shapes, which allows this distribution to fit different types of lifetime data.

\section{Statistical Properties}
In this section, we study the statistical properties for the WGED, specially quantile function and simulation median, skewness, kurtosis and moments.

\subsection{Quantile and median}
In this subsection, we determine the explicit formulas of the quantile and the median of WGED distribution. The quantile $x_q$ of the WGED is given by

\begin{equation} \label{eq3.1}
F(x_q)=q, \quad 0<q<1.
\end{equation}

\noindent
From Eq. (\ref{eq2.1}), $x_q$ can be obtained as follows.

\begin{equation} \label{eq3.2}
x_q= \frac{1}{\lambda} \ln \left[ 1+\left(\frac{1}{a} \ln (1-q) \right)^{\frac{1}{b}} \right].
\end{equation}

\noindent
Setting $q= 0.5$ in Eq. (\ref{eq3.2}), we get the median of WGED as follows.
\begin{equation} \label{eq3.3}
x_q= \frac{1}{\lambda} \ln \left[ 1+\left(- \frac{\ln (2)}{a}  \right)^{\frac{1}{b}} \right].
\end{equation}

\subsection{The mode}
In this subsection, we will derive the mode of the WGED by derivation its pdf with respect to $x$ and equate it to zero. The mode is the solution the following equation with respect to $x$.

\begin{equation} \label{eq3.4}
f^{'}(x)=0.
\end{equation}

\noindent
By substitution PDF from Eq. (\ref{eq2.2}) in Eq.(\ref{eq3.4}), we have

\begin{eqnarray*}
&&
\frac{\partial }{\partial x} \bigg [ ab\lambda \cdot  e^{\lambda x_i} \cdot [e^{\lambda x_i}-1]^{b-1}\cdot  e^{-a[e^{\lambda x_i}-1]^b} \bigg ] = 0,
\\
&&
\frac{\partial }{\partial x} \bigg [ h(x;a,b,\lambda ) \cdot S(x;a,b,\lambda )\bigg ]=0,
\\
&&
\bigg [ h^{'}(x;a,b,\lambda )\cdot S(x;a,b,\lambda )+ h(x;a,b,\lambda )\cdot S^{'}(x;a,b,\lambda ) \bigg ] =0,
\end{eqnarray*}
then
\begin{equation}\label{eq3.5}
\bigg [ h^{'}(x;a,b,\lambda )- (h(x;a,b,\lambda ))^2 \bigg ]S(x;a,b,\lambda ) =0,
\end{equation}

\noindent
where $h(x;a,b,\lambda), \, S(x;a,b,\lambda )$ are hazard function and survival function of WGED respectively.
\\
It is not possible to get an analytic solution in $x$ to Eq. (\ref{eq3.5}) in the general case. It has to be obtained numerically by using methods such as fixed-point or bisection method.

\subsection{Skewness and kurtosis}
The analysis of the variability Skewness and Kurtosis on the shape parameters $b,\lambda $  can be investigated based on quantile measures. The short comings of the classical Kurtosis measure are well-known. The Bowely's skewness \cite{Kenney1962} based on quartiles is given by

\begin{equation} \label{eq3.6}
S_k=\frac{ q_{(0.75)} -2 q_{(0.5)}+q_{(0.25)}}{q_{(0.75)}- q_{(0.25)}},
\end{equation}

\noindent
and the Moors' Kurtosis \cite{Moors1998} is based on octiles

\begin{equation} \label{eq3.7}
K_u=\frac{ q_{(0.875)} - q_{(0.625)}-q_{(0.375)}+ q_{(0.125)}}{q_{(0.75)}- q_{(0.25)}},
\end{equation}

\noindent
where $q_{(.)}$ represents quantile function.

\subsection{Moments}
In this subsection, we discuss the $rth$ moment for WGED. Moments are important in any statistical analysis, especially in applications. It can be used to study the most important features and characteristics of  a distribution (e.g. tendency, dispersion, skewness and kurtosis).

\begin{theorem} \label{Th1}
If $X$ has WGED $(a,b,\lambda )$, then The $r$th moments of random variable $X$, is given by the following

\begin{equation}  \label{eq3.8}
\mu_ r^{'}= \sum_{i=0}^{\infty}\sum_{j=0}^{\infty}\sum_{k=0}^{\infty}
\frac{(-1)^{i+k} a^{i+1}\Gamma (b(i+1)+j+1) \Gamma (r+1)} {i! j! \lambda^r (k+1)^{r+1} \Gamma ((b+i)+1 )}  \binom{b(i+1)+j+1}{k}.
\end{equation}
\end{theorem}

\begin{proof}
We start with the well known distribution of the $r$th moment of the random variable $X$ with probability density function $f(x)$ given by

\begin{equation} \label{eq3.9}
\mu_ r^{'}=\int_0^{\infty} x^r f(x;a,b,\lambda) dx.
\end{equation}

\noindent
Substituting from Eq. (\ref{eq1.1}) and Eq. (\ref{eq1.2}) into Eq. (\ref{eq1.13}) we get

\begin{equation*}
\mu_ r^{'}= \sum_{i=0}^{\infty}\sum_{j=0}^{\infty}\frac{(-1)^i a^{i+1} b\cdot  \Gamma (b(i+1)+j+1)}{i! j! \Gamma ((b+i)+1)} \int\limits_{0}^{\infty} x^r \lambda e^{-\lambda x} \left[1-e^{-\lambda x}\right]^{b(i+1)+j-1} dx,
\end{equation*}

\noindent
since $0<1-e^{-\lambda x}<1$ for $x>0$, the binomial series expansion of $$\left[1-e^{-\lambda x}\right]^{b(i+1)+j-1}$$ yields
\[
\left[1-e^{-\lambda x}\right]^{b(i+1)+j-1}=
\sum_{k=0}^{\infty} (-1)^k\cdot \binom{b(i+1)+j-1}{k } e^{-k \lambda x},
\]
then we get
\begin{equation*}
\mu_ r^{'}= \sum_{i=0}^{\infty}\sum_{j=0}^{\infty}\sum_{k=0}^{\infty} \frac{(-1)^{i+k} a^{i+1} \lambda  \Gamma (b(i+1)+j+1)}{i! j!\Gamma ((b+i)+1)}\binom{ b(i+1)+j-1 }{ k} \int_{0}^{\infty} x^r e^{-(k +1)\lambda x} dx,
\end{equation*}

\noindent
by using the definition of gamma function in the form, Zwillinger \cite{Zwillinger},
\[
\Gamma (Z)=x^z \int\limits_{0}^{\infty}e^{tx} t^{z-1}dt, \quad z,x,>0.
\]
Finally, we obtain the $r$th moment of WGED in the form

\begin{equation*}
\mu_ r^{'}= \sum_{i=0}^{\infty}\sum_{j=0}^{\infty}\sum_{k=0}^{\infty} \frac{(-1)^{i+k} a^{i+1} \Gamma (b(i+1)+j+1) \Gamma (r+1)}{i! j! \lambda^{r} (k+1)^{r+1} \Gamma ((b+i)+1 )} \binom{b(i+1)+j+1}{k}.
\end{equation*}
This completes the proof.
\end{proof}

\section{The Moment Generating Function}
The moment generating function (mgf),  $M_X(t)$,  of a random variable $X$ provides the basis of an alternative route to analytic results compared with working directly with the pdf and cdf of $X$.

\begin{theorem} \label{Th2}
The moment generating function (mgf) of WGED is given by

\begin{equation} \label{eq4.1}
M_X(t) = \sum_{r=0}^{\infty}\sum_{i=0}^{\infty}\sum_{j=0}^{\infty}\sum_{k=0}^{\infty}
\frac{(-1)^{i+k} a^{i+1} t^r \Gamma (b(i+1)+j+1) \Gamma (r+1)}{i! j!r! \lambda^{r} (k+1)^{r+1} \Gamma ((b+i)+1 )} \binom{b(i+1)+j+1}{k}.
\end{equation}
\end{theorem}

\begin{proof}
The moment generating function,  $M_X(t)$,  of the random variable $X$ with probability density function, $f(x)$ is given by

\begin{equation*}
M_X(t) = \int_{0}^{\infty}e^{tx} f(x;a,b,\lambda  ) dx,
\end{equation*}

\noindent
using series expansion of $e^{xt}$, we obtain

\begin{equation} \label{eq4.mg1}
M_X(t) = \sum_{r=0}^\infty  \frac{t^r}{r!} \int_0^{\infty} x^r f(x;a,b,\lambda  ) dx
= \sum_{r=0}^\infty  \frac{t^r}{r!} \mu_r^{'}.
\end{equation}

\noindent
Substituting from Eq. \ref{eq3.8} into Eq. \ref{eq4.mg1}, we obtain the moment generating function of WGED in the form

\begin{equation*}
M_X(t) = \sum_{r=0}^{\infty}\sum_{i=0}^{\infty}\sum_{j=0}^{\infty}\sum_{k=0}^{\infty}
\frac{(-1)^{i+k} a^{i+1} t^r \Gamma (b(i+1)+j+1) \Gamma (r+1)}{i! j!r! \lambda^{r} (k+1)^{r+1} \Gamma ((b+i)+1 )} \binom{b(i+1)+j+1}{k}.
\end{equation*}
This completes the proof.
\end{proof}

\section{Order Statistics }
In this section, we derive closed form expressions for the pdf  of the $r$th order statistic of the WGED. Let $X_{1:n}, X_{2:n},\cdots, X_{n:n}$ denote the order statistics obtained from a random sample $X_1, X_2, \cdots, X_n$ which taken from a continuous population with cumulative distribution function (cdf), $F(x;\varphi)$ and probability density function (pdf), $f(x;\varphi )$, then the probability density function of $X_{r:n}$ is given

\begin{equation} \label{eq5.1}
f_{r:n}(x;\varphi )=\frac{1}{B(r,n-r+1)}[F(x;\varphi )]^{r-1} [1-F(x;\varphi )]^{n-r} f(x;\varphi),
\end{equation}

\noindent
where $f(x;\varphi)$, $F(x;\varphi)$ are the pdf and cdf of $WGED(\varphi )$ given by Eq. (\ref{eq2.2}) and Eq.(\ref{eq2.1}) respectively, $\varphi =(a, b, \lambda )$ and $B(.,.)$ is the beta function, also we define first order statistics $X_{1:n}= \min (X_1, X_2, \cdots, X_n)$, and the last order statistics as $X_{n:n}= \max (X_1, X_2, \cdots, X_n)$. Since $0 < F(x;\varphi )< 1$  for $x>0$, we can use the binomial expansion of $[1-F(x;\varphi )]^{n-r}$ given as follows

\begin{equation} \label{eq5.2}
[1-F(x;\varphi )]^{n-r}=\sum_{i=0}^{n-r} \binom{n-r}{ i} (-1)^i [F(x;\varphi )]^i.
\end{equation}

\noindent
Substituting from Eq. (\ref{eq5.2}) into Eq. (\ref{eq5.1}), we obtain

\begin{equation} \label{eq5.3}
f_{r:n}(x;\varphi )=\frac{f(x;\varphi )}{B(r,n-r+1)} \sum_{i=0}^{n-r} \binom{ n-r}{i} (-1)^i [F(x;\varphi )]^{i+r-1}.
\end{equation}

\noindent
Substituting from Eq. (\ref{eq2.1}) and Eq. (\ref{eq2.2}) into Eq. (\ref{eq5.3}), we obtain

\begin{equation} \label{eq5.4}
f_{r:n}(x;a,b,\lambda )=\sum_{i=0}^{n-r} \sum_{j=0}^{i+r-1}
\frac{(-1)^{i+j} n!}{i! (r-1)! (n-r-1)!} \binom{ i+r-1}{ j } f(x;(j+1)a,b,\lambda ).
\end{equation}

\noindent
Relation (\ref{eq5.4}) show that $f_{r:n}(x;\varphi )$ is the weighted average of the Weibull-G exponential distribution with different shape parameters.

\section{Parameters Estimation}
In this section, point and interval estimation of the unknown parameters of the WGED are derived by using the method of maximum likelihood based on a complete sample data.

\subsection{Maximum likelihood estimation:}
Let $x_1, x_2, \cdots, x_n$ denote a random sample of complete data from the WGED. The likelihood function is given as

\begin{equation} \label{eq6.1}
L = \prod_{i=1}^{n} f(x_i, a, b, \lambda ),
\end{equation}

\noindent
substituting from (\ref{eq2.2}) into (\ref{eq6.1}), we have

\begin{equation*}
L = \prod_{i=1}^{n}ab\lambda \,  e^{\lambda x_i} [e^{\lambda x_i}-1]^{b-1}  e^{-a[e^{\lambda x_i}-1]^b}.
\end{equation*}

\noindent
The log-likelihood function is

\begin{equation} \label{eq6.2}
\mathcal{L} = n \ln (ab\lambda ) + \lambda \sum_{i=1}^{n} x_i +(b-1)\sum_{i=1}^{n}\ln (e^{\lambda x_i}-1) -a \sum_{i=1}^{n} [e^{\lambda x_i}-1]^b.
\end{equation}

\noindent
The maximum likelihood estimation of the parameters are obtained by differentiating the log-likelihood function $\mathcal{L}$ with respect to the parameters $a, b$ and $\lambda$ and setting the result equal to zero, we have the following normal equations.

\begin{eqnarray} \label{eq6.3}
\frac{\partial \mathcal{L} }{\partial a} &=&
\frac{n}{a}- \sum_{i=1}^{n} \left[e^{\lambda x_i}-1\right]^b= 0,
\\ \label{eq6.4}
\frac{\partial \mathcal{L}}{\partial b} & = &
\frac{n}{b}+ \sum_{i=1}^{n} \ln \left(e^{\lambda x_i}-1\right)- a \sum_{i=1}^{n} \left[e^{\lambda x_i}-1 \right]^b
\ln \left(e^{\lambda x_i}-1\right)= 0,
\\ \label{eq6.5}
\frac{\partial \mathcal{L}}{\partial \lambda } & = &
\frac{n}{\lambda} + \sum_{i=1}^n x_i  + (b-1) \sum_{i=1}^n \frac{x_i e^{\lambda x_i}}{e^{\lambda x_i}-1} -
a b \sum_{i=1}^n x_i e^{\lambda x_i} \left[e^{\lambda x_i}-1\right]^{b-1}  = 0.
\end{eqnarray}

\noindent
The MLEs can be obtained by solving the nonlinear equations previous, (\ref{eq6.3})--(\ref{eq6.5}), numerically for $a, b$ and
$\lambda$.


\subsection{Asymptotic confidence bounds}
In this section, we derive the asymptotic confidence intervals of these parameters when $a, b >0$ and $\lambda>0$ as the MLEs of the unknown parameters $a, b >0$ and $\lambda >0$ can not be obtained in closed forms, by using variance covariance matrix $I^{-1}$ see Lawless \cite{Lawless2003}, where $I^{-1}$ is the inverse of the observed information matrix which defined as follows

\begin{eqnarray}\label{eq6.6}
\mathbf{I^{-1}} & = &
\left(
\begin{array}{ccc}
-\frac{\partial ^2 \mathcal{L}}{\partial a ^2} & -\frac{\partial ^2 \mathcal{L}}{\partial a \partial b} & -\frac{\partial ^2 \mathcal{L}}{\partial a \partial \lambda  }
\\
-\frac{\partial ^2 \mathcal{L}}{\partial b \partial a } & -\frac{\partial ^2 \mathcal{L}}{\partial b^2} & -\frac{\partial ^2 \mathcal{L}}{\partial b \partial \lambda }
\\
-\frac{\partial ^2 \mathcal{L}}{\partial \lambda  \partial a} & -\frac{\partial ^2 \mathcal{L}}{\partial \lambda  \partial b} & -\frac{\partial ^2 \mathcal{L}}{\partial \lambda ^2}
\end{array}
\right)^{-1}
\nonumber\\
& =&
\left(
\begin{array}{ccc}
var(\hat{a}) & cov( \hat{a}, \hat{b}) & cov( \hat{a}, \hat{ \lambda  })
\\
cov( \hat{b},\hat{a }) & var( \hat{b}) & cov( \hat{b}, \hat{ \lambda  })
\\
 cov( \hat{ \lambda  }, \hat{a}) & cov( \hat{ \lambda }, \hat{ b}) &  var( \hat{ \lambda })
\end{array}
\right).
\end{eqnarray}

\noindent
The second partial derivatives included in $I$ are given as follows.

\begin{eqnarray} \label{eq6.7}
\frac{\partial ^2 \mathcal{L}}{\partial a^2} & = & - \frac{n}{a^2},
\\ \label{eq6.8}
\frac{\partial ^2 \mathcal{L}}{\partial a \partial b } & = &
- \sum_{i=1}^{n}\left[e^{\lambda x_i}-1\right]^b \ln \left( e^{\lambda x_i}-1\right),
\\ \label{eq6.9}
\frac{\partial ^2 \mathcal{L}}{\partial a \partial \lambda  } &= &
-b \sum_{i=1}^{n}x_i e^{\lambda x_i} \left[e^{\lambda x_i}-1\right]^{b-1} ,
\\ \label{eq6.10}
\frac{\partial ^2 \mathcal{L}}{\partial b^2} & = &
- \frac{n}{b^2} -a \sum_{i=1}^{n} \left[e^{\lambda x_i}-1\right]^b \left[ \ln \left( e^{\lambda x_i}-1 \right) \right]^2,
\\ \label{eq6.11}
\frac{\partial ^2 \mathcal{L}}{\partial b \partial \lambda }& = &
\sum_{i=1}^{n}\frac{x_i e^{\lambda x_i}}{e^{\lambda x_i}-1} - a \sum_{i=1}^{n} x_i e^{\lambda x_i} [e^{\lambda x_i}-1]^{b-1}
\left[ b \ln (e^{\lambda x_i}-1) +1 \right] ,
\\ \label{eq6.12}
\frac{\partial ^2 \mathcal{L}}{\partial \lambda  ^2 } & = & - \frac{n}{\lambda ^2}
-(b-1)\sum_{i=1}^{n}\frac{x_i^2 e^{\lambda x_i}}{(e^{\lambda x_i}-1)^2}
- a b \sum_{i=1}^{n}
x_i^2 e^{\lambda x_i} \left(b e^{\lambda x_i}-1 \right ) \left[e^{\lambda x_i}-1\right]^{b-2} .
\end{eqnarray}

\noindent
We can derive the $(1-\delta)100\%$ confidence intervals of the parameters $a, b$ and $\lambda $, by using variance matrix as in the following forms
\[
\hat{a} \pm Z_{\frac{\delta}{2}}\sqrt{var(\hat{a})},\quad \hat{b} \pm Z_{\frac{\delta}{2}}\sqrt{var(\hat{b})},\quad \hat{\lambda } \pm Z_{\frac{\delta}{2}}\sqrt{var(\hat{\lambda  })},
\]

\noindent
where $Z_{\frac{\delta}{2}}$ is the upper $(\frac{\delta}{2})$-th percentile of the standard normal distribution.


\section{Application}
In this section, we present the analysis of a real data set using the WGED $(a, b, \lambda )$ model and compare it with the other fitted models such as exponential distributions (ED), generalized exponential distribution (GED), \cite{Gupta1999}, beta exponential distribution (BED), \cite{Nadarajah2006} and the beta generalized exponential distribution (BGED), \cite{Wagner2009} using Kolmogorov–Smirnov (K–-S) statistic, as well as Akaike information criterion (AIC), Akaike , \cite{Akaike1974}, Bayesian information criterion (BIC) and Hannan-Quinn information criterion (HQIC) values, Schwarz \cite{Schwarz1978}.

\noindent
The data set is obtained from Smith and Naylor \cite{Smith1987}. The data are the strengths of 1.5 cm glass fibres, measured at the National Physical Laboratory, England. Unfortunately, the units of measurement are not given in the paper. This data set is in Table 1.

\begin{center}
Table 1: The data are the strengths of 1.5 cm glass fibres, \cite{Smith1987}.\\
\begin{tabular}{ccccccccccc}\hline
0.55 & 0.93 & 1.25 & 1.36 & 1.49 & 1.52 & 1.58 & 1.61 & 1.64  \\
1.68 & 1.73 & 1.81 & 2    & 0.74 & 1.04 & 1.27 & 1.39 & 1.49  \\
1.53 & 1.59 & 1.61 & 1.66 & 1.68 & 1.76 & 1.82 & 2.01 & 0.77\\
1.11 & 1.28 & 1.42 & 1.5  & 1.54 & 1.6  & 1.62 & 1.66 & 1.69 \\
1.76 & 1.84 & 2.24 & 0.81 & 1.13 & 1.29 & 1.48 & 1.5  & 1.55 \\
1.61 & 1.62 & 1.66 & 1.7  & 1.77 & 1.84 & 0.84 & 1.24 & 1.3 \\
1.48 & 1.51 & 1.55 & 1.61 & 1.63 & 1.67 & 1.7  & 1.78 & 1.89  \\ \hline	 	
\end{tabular}
\end{center}

\noindent
Table 2 gives MLs of parameters of the WGED and sub-models and goodness of fit statistics are in Table 3.

\vspace{0.4 cm}
\begin{center}
Table 2: MLEs of parameters, Log-likelihood.\\
\begin{tabular}{ccccccc} \hline
Model   & MLEs of parameters &  K-S & p-value\\ \hline
ED     &   $\hat{\lambda }$=0.664                                        & 0.402316 & $1.44529\times 10^{-09}$\\
GED    &  $\hat{\lambda}$=2.6105, \; $\hat{\alpha }$=31.3032             & 0.213118 & 0.005444 \\
BED    &  $\hat{a}$=17.7786, \; $\hat{b}$=22.7222, \; $\hat{\lambda }$=0.3898 & 0.159819 & 0.07220 \\
BGED   &  $\hat{a}$=0.4125, \; $\hat{b}$=93.4655, \; $\hat{\lambda }$=0.92271, \; $\hat{\alpha }$=22.6124 & 0.150611 & 0.10470\\
WGED  &   $\hat{a}$=56.881, \; $\hat{b}$=4.893, \; $\hat{\lambda }$=0.222  &   0.127366 & 0.24259\\ \hline
\end{tabular}
\end{center}

\vspace{0.4 cm}

\begin{center}
Table 3: Log-likelihood, AIC, AICC, BIC and HQIC values of models fitted.
\begin{tabular}{lccccccc} \hline
Model  & $\mathcal{L}$  &  $-2\mathcal{L}$& AIC     & AICC     & BIC & HQIC \\ \hline
ED	&	-88.8300	&	-177.6600	&	179.6600	&	179.7256	&	181.8031	&	180.5029	\\
GED	&	-31.3834	&	-62.7668	&	66.7668	&	66.9668	&	71.0531	&	68.4526	\\
BED	&	-24.1270	&	-48.2540	&	54.2540	&	54.6608	&	60.6834	&	56.7827	\\
BGED	&	-15.5995	&	-31.1990	&	39.1990	&	39.8887	&	47.7715	&	42.5706	\\
WGED	&	-14.828	&	-29.6560	&	35.6560	&	36.0628	&	42.0854	&	38.1847	\\ \hline
\end{tabular}
\end{center}

\noindent
We find that the WGED distribution with the three-number of parameters provides a better fit than the previous new modified exponential distribution which was the best in \cite{Guptar12001,Gupta1999,Nadarajah2006,Wagner2009}. It has the largest likelihood, and the smallest K-–S, AIC, BIC and HQIC values among those considered in this paper.

\noindent
Substituting the MLE's of the unknown parameters $a, b, \lambda $  into (\ref{eq6.6}), we get estimation of the variance covariance matrix as the following

$$
I_0^{-1}=\left(
\begin{array}{cccr}
3.655\times 10^{3}  &	7.228               & -2.205             \\
7.228 	            &   0.213               & 1.141\times 10^{-3} \\
-2.205              & 	1.141\times 10^{-3} & 1.505\times 10^{-3}  \\
\end{array}
\right)
$$

\noindent
The approximate 95\% two sided confidence intervals of the unknown parameters $a, b$  and $\lambda  $ are \\
$\left[0 ,175.11\right]$, $\left[3.989,	57.785 \right]$ and $\left[0.146, 0.298 \right]$,  respectively.

\noindent
To show that the likelihood equation have unique solution, we plot the profiles of the log-likelihood function of $a,  b$ and  $\lambda $ in Figures 7 and 8.

\begin{center}
\includegraphics[width=7.2cm,height=6.2cm]{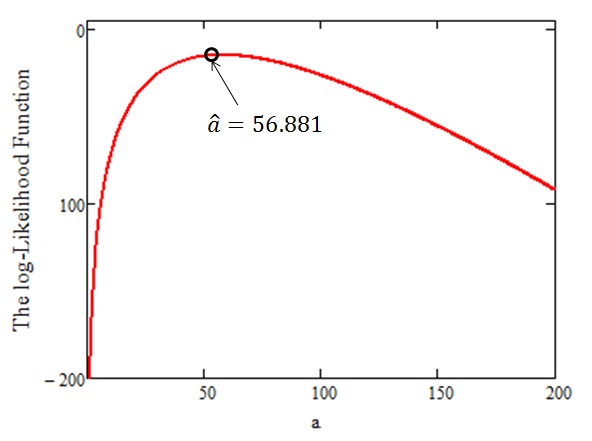}
\includegraphics[width=7.2cm,height=6.2cm]{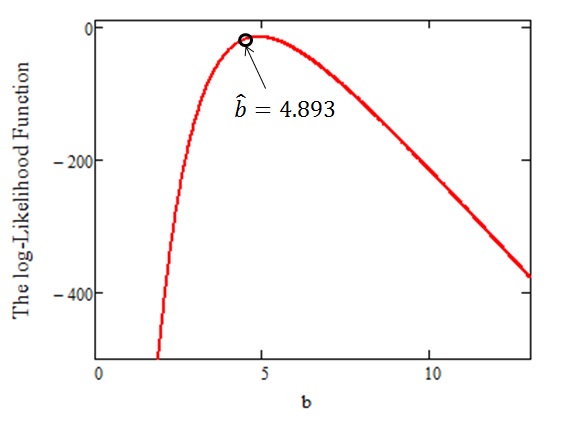} \\
Figure 7: The profile of the log-likelihood function of $a, b$.
\end{center}

\vspace{0.2 cm}

\begin{center}
\includegraphics[width=10cm,height=5.5cm]{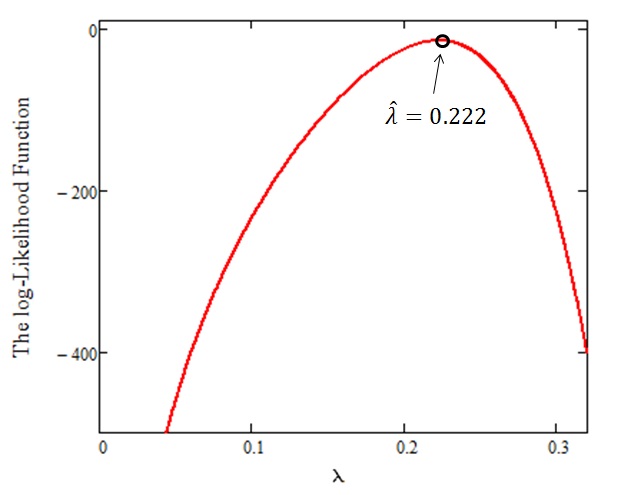}\\
Figure 8: The profile of the log-likelihood function of $\lambda$.
\end{center}

\noindent
The nonparametric estimate of the survival function using the Kaplan-Meier method and its fitted parametric estimations when the distribution is assumed to be $ED, GE, BED, BGED$ and $WGED$ are computed and plotted in Figure 9.

\begin{center}
\includegraphics[width=10cm,height=6cm]{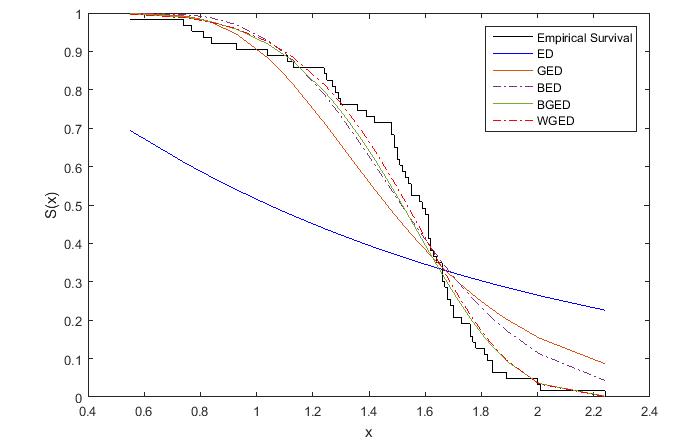}\\
Figure 9: The Kaplan-Meier estimate of the survival function for the data.
\end{center}

\noindent
Figure 10 gives the form of the CDF for the $ED, GE, BED, BGED$ and $WGED$   which are used to fit the data after replacing the unknown parameters included in each distribution by their MLEs.

\begin{center}
\includegraphics[width=10cm,height=6cm]{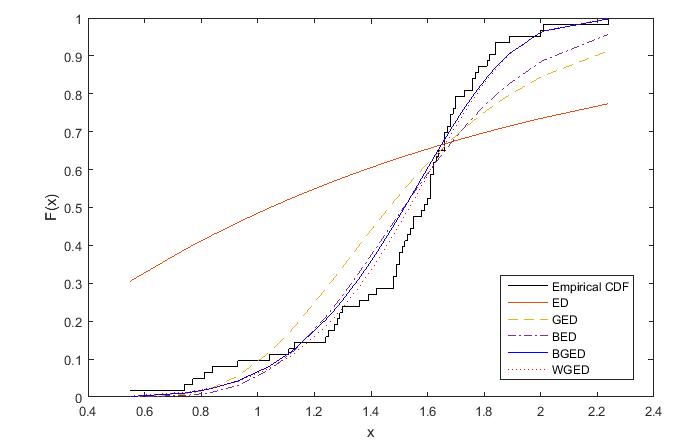}\\
Figure 10: The Fitted cumulative distribution function for the data.
\end{center}

\section{Conclusions}
A new distribution, based on  Weibull- G family distributions, has been proposed and its properties studied. The idea is to add parameter to exponential distribution, so that the hazard function is either increasing or more importantly, bathtub shaped. Using  Weibull  generator component, the distribution has flexibility to model the second peak in a distribution. We have shown that the Weibull-G exponential distribution fits certain well-known data sets better than existing modifications of the generalized families of exponential probability distribution.



\end{document}